\documentclass[12pt,twoside]{article}
\usepackage{amsmath,amsfonts,amsfonts,amssymb}
\usepackage{amsbsy,theorem}
\usepackage{mathrsfs}
\usepackage{amscd}
\usepackage{pstricks}
\usepackage{color}
\usepackage{graphicx}

\theoremstyle{break}
\newtheorem{lemma}{Lemma}[section]

\newtheorem{proposition}[lemma]{Proposition}
\newtheorem{theorem}[lemma]{Theorem}
\newtheorem{corollary}[lemma]{Corollary}
\newtheorem{conjecture}[lemma]{Conjecture}

\theorembodyfont{\upshape}
\newtheorem{remark}[lemma]{Remark}
\newtheorem{definition}[lemma]{Definition}
\newtheorem{example}[lemma]{Example}

\newcommand \QQ {{\mathbb Q}}

\newcommand \FF {{\mathbb F}}

\newcommand \AAA {{\mathbb A}}

\newcommand \ZZ {{\mathbb Z}}
\newcommand \LL {{\mathbb L}}
\newcommand \cI {{\mathcal I}}

\newcommand \cV {{\mathcal V}}

\newcommand{\q}{/\!\!/}

\newcommand{\rank}{\mathrm{rank\,}}

\newcommand{\girth}{\mathrm{girth}\,}

\newcommand{\Mat}{\mathrm{Mat}}
\newcommand{\adj}{\mathrm{adj}\,}

\newcommand\ph{\scalebox{1.3}{$\varphi$}}

\newcommand{\1}{{\rm 1\hspace*{-0.4ex}%
\rule{0.1ex}{1.52ex}\hspace*{0.2ex}}}

\begin{document}
\begin{center}
\LARGE\textsf{Dual Graph Polynomials and a 4-face Formula}\\
\bigskip
\normalsize \textsc{ Dmitry Doryn}\\
doryn@mpim-bonn.mpg.de\\
\end{center}

\begin{abstract}
We study the dual graph polynomials $\ph_G$ and the case when a Feynman graph has no triangles but has a 4-face. This leads to the proof of the duality-admissibility of all graphs up to 18 loops. As a consequence, the $c_2$ invariant is the same for all 4 Feynman period representations (position, momentum, parametric and dual parametric) for any physically relevant graph.
\end{abstract}
\section{Introduction}
The analysis of amplitudes and periods in renormalization group functions by
means of arithmetic and algebraic geometry has become a common quest in recent
years. Since the work of Broadhurst and Kreimer, \cite{BrKr}, it is well-known that the single-scale massless Feynman integral in perturbative quantum field theory usually give rise to interesting patterns involving multiple zeta values (MZV). In particular, the Feynman periods for primitive graphs in $\phi^4$ are evaluated to elements in $\QQ$-algebra of MZV for almost all known cases, see \cite{Sch}. The first (and, so far, unique) example, when a Feynman period gives something worse, was computed by Panzer in \cite{Pa}, the value is expressible in terms of multiple polylogarithms evaluated at primitive sixth roots of unity. Unfortunately, these values are obtained by the intensive numerical analysis and there is no good way to predict the periods of Feynman graphs in general.

The first step to the understanding of the Feynman period from the algebro-geometrical perspective was done by Bloch, Esnault and Kreimer in \cite{BEK}, where the "Feynman motive" was defined. Further results in the cohomological direction can be found in \cite{D}, \cite{BrD}. More can be done on the arithmetical side, see \cite{St}, \cite{D2}, \cite{Sch2}, \cite{BrSch}. Out of the number of rational points on the poles of the Feynman differential form, one can define the $c_2$ invariant. The miracle is that it respects all the relations between known periods, so it seems to be the discrete analogue of the Feynman period. In this article we continue to study the properties of the $c_2$ invariant.   \medskip

For a graph $G$, define the graph polynomial and the dual graph polynomial
\begin{equation}\label{d1}
\Psi_G:=\sum_{sp.tr. T}\prod_{e\not\in T} \alpha_e, \quad \ph_G:=\sum_{sp.tr. T}\prod_{e\in T} \alpha_e\in \ZZ[\alpha_1,\ldots,\alpha_{N_G}]
\end{equation}
with the sums going over all spanning trees. The variety $X_G:=\cV(\Psi_G)\subset \AAA^{N_G}$ describes the poles (of order 2) of the Feynman differential form. For being able to speak on the Feynman period one needs to restrict to log-divergent graphs: the graphs $G$ with the number of edges equal to twice the loop number, $N_G=2h_G$.

Counting the $\FF_q$-rational points of the graph hypersurface $X_G$, one observes that the important piece of this value is the coefficient of $q^2$ in the $q$-expansion:
\begin{equation}\label{b2}
c_2(G)_q:= \#X_G(\FF_q)/q^2 \mod q^3.
\end{equation}
It is called the $c_2$ invariant. On one side, we are able to compute this coefficient analytically (or partially on a computer) for many small (physically relevant) graphs. On the other side, it turns out the this coefficient contains some information about the period.

There are 4 different representations of the Feynman period: in position and momentum spaces, parametric and dual parametric representations. The 4 resulting values do coincide. One can also try to get a discrete analogue of this result. In \cite{BSY}, the authors have constructed the $c_2$ invariant $c_2(G)^{mom}_q$ out of the geometry of the poles of the Feynman period in momentum space and have proved that $c_2(G)^{mom}_q=c_2(G)_q$ for log-divergent graphs. In \cite{D4}, the $c_2$ invariants $c_2(G)^{pos}_q$ and $c_2(G)^{dual}_q$ were defined in position space and in dual parametric space out of the related geometry, and the coincidence of all four $c_2$ invariants was proved for graphs with minor conditions plus the important restriction to the graphs called duality admissible:
\begin{theorem}\label{thm11}
Let $G$ be a log-divergent graph that is duality admissible with $h_G\geq 3$. Then 
\begin{equation}\label{b3}
c_2(G)^{mom}_q=c_2(G)_q=c_2(G)^{dual}_q=c_2(G)^{pos}_q.
\end{equation} 
\end{theorem}
The condition of duality admissibility for $G$ means the vanishing of $c_2(\gamma)^{dual}_q$ for certain sub-quotient graphs $\gamma$ of $G$ (see Definition \ref{def36}). This condition is the property that is surprisingly hard to verify in general, but seems to be always satisfied.
\begin{conjecture}
Let $G$ be a log-divergent graph with $h_G\geq 3$. Then $G$ is duality admissible.
\end{conjecture}
In \cite{D4}, the conjecture was verified for graphs $G$ with $\girth(G)\leq 3$. Here $\girth(G)$ is the minimal $n$ such that each cycle of $G$ has length at least $n$. While checking the duality admissibility we should control a half of all sub-quotient graphs of $G$, so the case $\girth=3$ is not enough already for several graphs with 7 loops.

In this article we prove the conjecture for graphs with $\girth(G)=4$. More precisely (Theorem \ref{thm53}):
\begin{theorem}
Let $G$ be a log-divergent graph with $3\leq h_G\leq 18$ loops. Then $G$ is duality admissible. 
\end{theorem}

Hence, for all these graphs (\ref{b3}) holds (see Theorem \ref{thm5n4}). The indication of the bound $h_G\leq 18$ comes from the fact that  the first minimal log-divergent graph with $\girth=5$ has 18 vertices. The Feynman periods are computed only for graphs up to 8 loops (and for several 9-loop graphs), as well as for the several infinite series of graphs like $WS_n$, $ZZ_n$, which have $\girth=3$. Thus, we cover all the interesting Feynman graphs so far. On the other hand, the graphs with $\girth(G)=4$ enter the game since, for example, one of the first counter-examples to Kontsevich conjecture on the number of rational points on graph hypersurfaces was a graph with 7 loops and girth 4, see \cite{D2}, \cite{Sch2}. In addition to Theorem above, we formulate a sufficient combinatorial criterion for an arbitrary graphs ($h_G\geq 3$) to be duality admissible, see Theorem \ref{thm5n5}

In Section 2, we introduce a new algebraic way of understanding the dual graph polynomials $\ph_G$: we do not use the Dodgson polynomials for $\Psi_G$ with inverted variables (Cremona transformation), but we introduce $\ph_G$ and the dual Dodgson polynomials as minors of a certain matrix $L_G$. This leads to a better control of the sings in the formulas and to an independent picture of dual graph polynomials situation from that one of the graph polynomials. 

The computational technique is presented in Section 3, as well as the known or intuitive results related to graphs with triangles. The proved facts are very similar to the case of the graphs hypersurface itself.  
We work in the Grothendieck ring of varieties $K_0(Var_k)$ and then jump to the computation for the number of $\FF_q$-rational points since we are going to intensively use the Chevalley-Warning vanishing. The most complicated and technical computations explaining the 4-face situation are moved to Section 4. 

The main result is stated in Section 5.

\underline{Acknowledgements}: I would like to thank MPIM Bonn for hospitality and for the financial support.

\section{Dual graph polynomials}
From some point, studying the dual graph polynomials, we follow the strategy of Section 2 of \cite{Br} and prove the corresponding statements to the theorems for graph polynomials given there. We usually identify a graph with it's set of edges.

Consider a connected graph $G$. For the two free $\ZZ$-modules labelled with the set of edges $E=E(G)$ and the set of vertexes $V=V(G)$, define the map $\partial:\ZZ^E \rightarrow \ZZ^V$ : $ e \mapsto v_t-v_s$, where $v_s$ and $v_t$ are the source and the target of the edge $e$ respectively. We extending the map by linearity and get a homological sequence 
\begin{equation}
0\rightarrow H_1(G,\ZZ)\rightarrow \ZZ^E\overset{\partial}{\rightarrow}\ZZ^V\rightarrow H_0(G,\ZZ)\rightarrow 0.
\end{equation}
\begin{definition}\label{d4}
We call a set $C=\{c_1,\ldots c_h\}$, $c_i\in H_1(G,\ZZ)$ the \emph{basis of small cycles} of $G$, if the following conditions are satisfied:
\begin{description}
\item{i).} Each $c_i$ is a pre-image of an (oriented) cycle (topological loop).
\item{ii).} The set of $c_i$s generates $H_1(G,\ZZ)$.
\item{iii).} if $c_i+\sum_{j\neq i} \lambda_j c_j=\lambda c$ 
for some $i\leq h$, $c\in H_1(G,\ZZ)$, $\lambda,\lambda_j\in\ZZ$, $\lambda\neq 0$, then $\lambda=\pm 1$. 
\end{description}
\end{definition}
Since $G$ is connected, it follows that $H_0(G,\ZZ)\cong  \ZZ$ and then $h_G=h_1(G):=\rank H_1(G,\ZZ)=N_G-|V|+1$ is called the loop number.
For a generating set $C$ of $H_1(G,Z)$ satisfying the conditions (i) and (ii) of the definition, we construct the following $h_G\times N_G -$matrix $F=F_C$: $F_{i,j}$ equals 1 if the edge $e_j$ belongs to the cycle $c_i$ and the orientation of the edge and that of $c_i$ coincide, and equals $-1$ if the orientations are different, and equals 0 in the case the edge does not belong to $c_i$.
\begin{lemma}
Fix a basis of small cycles $C$ of a graph $G$. Let $F_C(T)$ be the square matrix that we get from $F_C$ after deletion of the columns labelled by $T$, $T\subset E(G)$. Let $T$ be a set of $N_G-h_G$ edges of $G$. Then
\begin{equation}
\det F_C(T)= \left\{ \begin{aligned}
\pm &1& \text{if T is a spanning tree},\\ &0&\text{otherwise}.
\end{aligned}\right.
\end{equation}
\end{lemma}
\begin{proof}
We fix a subgraph $T$ with $N_G-h_G$ edges. Doing the elementary row operations (over $\ZZ$) of the matrix $F_C$ we try to make $F_C(T)$ upper-triangular and can end up with one of the following three cases. \\
1) The matrix $F_C(T)$ (after the possible interchange of rows) becomes an upper-triangular matrix, the rows of $F_C(T)$ are linearly independent (over $\ZZ$) with diagonal entries $\pm 1$. Then $\det F_C(T)=\pm 1$. Assume that $T$ is not a spanning tree. Since $T$ has cardinhality  $N_G-h_G=|V|-1$, it follows that it is not a tree and has a loop $c'\subset T$. This loop gives us an element of $H_1(G,\ZZ)$ linearly independent of the rows of $F_C(T)$, this contradicts the assumption on the rank of $H_1(G,\ZZ)$. \\
2) The rows of $F_C(T)$ are linearly dependent. Then $\det F_C(T)=0$. Since $C$ generates $H_1(G,\ZZ)$, there is a linear combination $\sum \lambda_i e_i=0$ with not all coefficients equal zero, where the summation goes over the edges of $T$. This is impossible in the case $T$ is a spanning tree. To see this, consider a leaf with a non-zero coefficient or a vertex with no non-zero coefficients of the vertices below, this vertex cannot cancel out with something else in the sum above, so the sum cannot lie in the kernel of $\partial$.\\
3) Consider now the case $F_C(T)$ is upper-triangular, but not all diagonal entries are equal to $\pm1$. We can assume that we have   $\lambda e-\sum \lambda_i e_i=0$ in $H_1(G,\ZZ)$ for $e_i\in T$, $\lambda_i\in \ZZ$ and for some $e\in E(G)\backslash T$, $\lambda\geq 2$. It follows that each $\lambda_i$ is divisible $\lambda$. Indeed, since $T$ is a spanning tree and $e\not\in T$, there is a path with endpoints the same as endpoints of $e$. The pre-image of this cycle together with  the relation above give us a linear relation between edges in $T$ if not all $\lambda_i$ are $\pm\lambda$, but then $T$ cannot be a spanning tree (see case (2)). Thus, our elementary transformation yields an element $\lambda c'$ for $c'\in H_1(G,\ZZ)$, this contradicts the choice of $C$ (part (iii) of the Definition \ref{d4}) and case (3) never happens. 
\end{proof}

\begin{proposition}
Let $G$ be a connected graph. Then there exist a basis of small cycles of $G$.
\end{proposition}
\begin{proof}
One way to construct a basis is the following. Fix a spanning tree $T$. As in part (3) of the previous lemma, for each edge $e\in E(G)\backslash T$ there exist a path $p(e)$ with endpoints exactly that of $e$ and consisting of the only edges of $T$. Then $e$ and $p(e)$ together form a cycle. In this way we construct a set $C$ of $h_G$ cycles. Building the matrix $F_C$, we see that $F_C(T)$ (modulo interchange of rows) is a diagonal matrix with entries $\pm 1$. Thus, rows of $F_C$ are linearly independent and satisfy parts (i) -- (iii) of Definition \ref{d4}, so $C$ is a basis of small cycles.
\end{proof}

From now on, for any given graph $G$, we choose and fix some basis of small cycles $C$, build a matrix $F_C$, and, omitting subscript $C$, write $F_G$ instead.  

We define 
\begin{equation}\label{d2}
L_G:=
\left(
\begin{array}{c|c}
\Delta(\alpha)& F_G^t\\
\hline
-F_G & 0 
\end{array}
\right) \in \Mat_{N_G+h_G,N_G+h_G}(\ZZ[\{\alpha_i\}_{i\in E(G)}]),
\end{equation}
where $\Delta(\alpha)$ is the diagonal matrix with entries $\alpha_1,\ldots,\alpha_{N_G}$. Here and later, we often identify edges with their indices $E(G)=\{1,\ldots,N_G\}$.
For a graph $G$, we write $G\backslash I$ (resp. $G\q J $) for the graph that we get after deletion (resp. contraction) of the edges of the set $I\subset E(G)$ (resp. $J\subset E(G)$).

\begin{proposition}\label{prop4}
Let $G$ be any connected graph.\\
\textbf{i)} For the dual graph polynomial defined by (\ref{d1}), one obtains 
\begin{equation}
\ph_G=\det L_G.
\end{equation} 
\textbf{ii)} One has the contraction-deletion formula
\begin{equation}
\ph_G=\ph^e_G\alpha_e+\ph_{G,e}
\end{equation}
for any edge labelled by $e$, where the coefficients are again the dual graph polynomials $\ph^e_G=\ph_{G \q e}$ and $\ph_{G,e}=\ph_{G\backslash e}$. The contraction of an edge $e$ corresponds to the determinant of the matrix $L_G$ after deletion of the $e$-th row and column, and deletion of an edge corresponds to setting $\alpha_e$ to zero:
\begin{equation}
\ph_{G \q e}=\det L_G(e,e),\quad \ph_{G\backslash e}=\det L_G|_{\alpha_e=0}.
\end{equation}
\end{proposition}
\begin{proof}
One computes 
\begin{equation}\label{c14}
\det L_G=\sum_{T\subset G} \prod_{i\in T} \alpha_i \det\!\left(
\begin{array}{c|c}
0& \!F^t(T)\!\\
\hline
\!-F(T)\! & 0 
\end{array}
\right)=\sum_{\!\!\!\!\!T\subset G,|T|=h} \prod_{i\in T}\alpha_i\det F(T)^2
\end{equation}
In the middle matrix for both cases $|T|>h$ and $|T|<h$ the rows of the matrix become linear dependent, thus the determinant is zero. For the remaining summands, where $|T|=h$, we apply lemma above:  $\det F(T)^2=1$ if $T$ is a spanning tree, and zero otherwise. The second statement of the proposition follows from the contraction-deletion formula and an observation that the determinant $\det L_G$ is linear in $\alpha_e$ with the corresponding coefficients. The second part of the theorem follows directly from (\ref{d1}).
\end{proof}

For a matrix $M$, we write $M(I,J)$ for the minor that we get after deletion of the rows indexed by the set $I$ and of the columns indexed by $J$. 
\begin{definition}
Let $I,J,K$ be subsets of edges of $G$ such that $|I|=|J|$. Define the \textit{dual Dodgson polynomial} to be
\begin{equation}
\ph^{I,J}_{G,K}:=\det L_G(I,J)|_{\{\alpha_e=0,k\in K\}}.
\end{equation}
\end{definition}

On easily sees that $\ph^{I,J}_{G,K}=\ph^{J,I}_{G,K}$ and $\deg \ph^{I,J}_{G,K}=N_G-h-|I|$. Using the propositon above, one also gets
\begin{equation}
\ph^{I,J}_{G\backslash B\q A, K}=\ph^{I\cup A,J\cup A}_{G,K\cup B}
\end{equation}
for any $A,B\subset E(G)$. Thus we usually consider the case $I\cap J =\emptyset$ and $K=\emptyset$.

\begin{proposition}
With the notation above, one gets
\begin{equation}
\ph^{I,J}_{G,K}=\sum_{T\subset G}(\pm)\prod_{e\in T}\alpha_e,
\end{equation}
where the sum goes over all subgraphs $T\subset G$ which are simultaneously spanning trees for both $G\backslash( K\cup I\backslash (I\cap J))\q J$ and $G\backslash( K\cup J\backslash (I\cap J))\q I$. In particular, every monomial in $\ph^{I,J}_{G,K}$ also occurs in both $\ph^{I,I}_{G,J\cup K}$
and $\ph^{J,J}_{G,I\cup K}$.
\end{proposition}
\begin{proof}
By passing to the minor $G\mapsto G\backslash (I\cap J)\q K$, we reduce to the case $I\cap J=\emptyset$ and $K=\emptyset$. Similar to (\ref{c14}), one computes 
\begin{multline}
\det(L_G(I,J))=\sum_{S\subset G\backslash (I\cup J)}\prod_{i\in S} \alpha_i \det \left(
\begin{array}{c|c}
0& \!F^t(S\cup I)\!\\
\hline
\!-F(S\cup J)\! & 0 
\end{array}
\right) = \\\sum_{S\subset G\backslash (I\cup J)}\pm \prod_{i\in S} \alpha_i \det F(S\cup I) \det F(S\cup J). 
\end{multline}
The term on the right survives iff $\det F(S\cup I)\neq 0 \neq F(S\cup J)$, thus, by Proposition \ref{prop4}, both $S\cup I$ and $S\cup J$ are spanning trees of $G$. Since $I\cap J = \emptyset$, $S$ should be a spanning tree for $G\backslash I \q J$, similar to  $G\backslash J \q I$. Conversely, such an $S$ gives $\det F(S\cup I)=\pm 1=F(S\cup J)$ by Proposition \ref{prop4}. 
\end{proof}
Later, we will drop the subscript $G$ in the dual graph and Dodgdon polynomials when $G$ is clear from the context. 

Recall the following Pl\"ucker identities: 
\begin{lemma}\label{lem7}
Let $M$ be an $N\times N$ symmetric matrix and let $i_1,\ldots, i_{2n}$ be distinct indices between $1$ and $N$. Then 
\begin{equation}
\sum_{k=n}^{2n} (-1)^k\det M(\{i_1,\ldots,i_{n-1},i_k\},\{i_n,\ldots,\hat{i}_k,\ldots,i_{2n}\}).
\end{equation}
\end{lemma}
\begin{proof}
See Lemma 27 in \cite{Br}. 
\end{proof}
The matrix $L_G$ becomes symmetric after multiplication of some rows by $-1$, this allows us to apply the lemma above and to obtain
\begin{equation}\label{c25}
\sum_{k=n}^{2n}(-1)^k \ph_G^{\{i_1,\ldots,i_{n-1},i_k\},\{i_n,\ldots,\hat{i}_k,\ldots,i_{2n}\}}=0.
\end{equation}

We will also use the Jacobi determinant formula, 
\begin{lemma}
Let $M=(a_{ij})$ be an invertible $N\times N$ matrix and let $\adj M=(A_{ij})$ denote the adjoint matrix of $M$, i.e. the transpose of the cofactors of $M$. Then for any $k$, $1\leq k\leq N$,
\begin{equation}\label{c27}
\det (A_{ij})_{k\leq i,j\leq N}=\det(M)^{N-k-1}\det (a_{ij})_{1\leq i,j\leq N}.
\end{equation} 
\end{lemma}
\begin{proof}
See Lemma 28 in \cite{Br}. 
\end{proof}
The special case of this lemma $k=N-2$ is attributed to L. C. Dodgson: for any $1\leq p<q\leq N$ and $1\leq r<s\leq N$
\begin{equation}\label{c30}
A_{p,r}A_{q,s}-A_{p,s}A_{q,r}=\det(M)\det M(pq,rs).
\end{equation}

\begin{proposition}
Let $G$ be a connected graph and let $I$,$J$ be two subsets of edges with $|I| = |J|$ and let $a,b,c,d\in E(G)\backslash I\cup J$ and $S:=I\cup J\cup \{a,b,c,d\}$. Then the (first) Dodgson identity is
\begin{equation}\label{c33}
\ph^{Ia,Jb}_S\ph^{Ic,Jd}_S-\ph^{Ia,Jd}_S\ph^{Ic,Jb}_S=\pm \ph^{I,J}_S\ph^{Iac,Jbd}_S
\end{equation}
with $+$ sign when $(a-c)(b-d)>0$, and $-$ sign otherwise.\\
Now let $I$ and $J$ be two subset of edges with $|J|=|I|+1$ and let $a,b,c\not\in I\cup J$, $S:=I\cup J \cup \{a,b,c\}$. Then the second Dodgson identity is
\begin{equation}
\ph^{Ia,J}_S\ph^{Ibc,Jc}_S-\ph^{Iac,Jc}_S\ph^{Ib,J}_S=\pm \ph^{Ic,J}_S\ph^{Iab,Jc}_S.
\end{equation}
\end{proposition}
\begin{proof}
The first part follows from (\ref{c30}) while the second part can be proved similarly to part (2) of Lemma 30 in \cite{Br}.  
\end{proof}

Consider the Cremona transformation $\iota:\ZZ[\alpha_1,\ldots,\alpha_n]\longrightarrow\ZZ[\alpha_1,\ldots,\alpha_n]$ defined by 
\begin{equation}
 \iota(f)(\alpha):=\prod_{1\leq j\leq n} \alpha_i f\Big(\frac{1}{\alpha_1},\ldots,\frac{1}{\alpha_n}\Big) 
\end{equation}
for any homogeneous polynomial $f$.

\begin{lemma}\label{lem10}
\textbf{i)} Let $S,K\subset E(G)$ are subsets of edges of $G$, $S\cap K=\emptyset$. Then 
\begin{equation}
\ph^S_{G,K}=\iota(\Psi^K_{G,S}).
\end{equation}
\textbf{ii)} Let $i,j\subset E(G)$ are two edges of $G$,  $i,j\notin S\cup K$. Then the Dodgson polynomial is related to the dual Dodgson polynomial by Cremona transformation up to a sign:
\begin{equation}
\ph^{Si,Sj}_{G,K}=\pm\iota(\Psi^{Ki,Kj}_{G,S}).
\end{equation} 
\end{lemma}
\begin{proof}
The part \textbf{(i)} follows from (\ref{d1}) and Proposition \ref{prop4}. The statement in \textbf{(ii)} we can reduce to the case $K=S=0$. By the first Dodgson identity (\ref{c33}), we have 
\begin{equation}
\ph^{i,i}_2\ph^{j,j}_1 \pm \ph \ph^{ij,ij}=(\ph^{i,j})^2.
\end{equation}
The same identity with the same sign holds for a graph polynomial $\Psi$. Applying the Cremona transformation to both sides and using part \textbf{(i)}, one gets $\ph^{i,j}=\pm\iota(\Psi^{i,j})$.
\end{proof}
\begin{proposition}\label{prop11}
Let $A$, $B$, $I\subset E(G)$ be tree subsets of edges of a graph $G$, where $|A|$=$|B|$, $I=\{i_1,\ldots,i_k\}$ and $I\cap (A\cup B)=\emptyset$. If $\ph^{A\cup I, B\cup I}=0$, then for each $t$, $1\leq t\leq k$, we have
\begin{equation}\label{c40}
\ph^{Ai_t,Bi_t}=\sum _{s\neq t}\pm\ph^{Ai_t,Bi_s} =\sum_{s\neq t}\pm\ph^{Ai_s,Bi_t}
\end{equation} 
as elements in $\ZZ[\alpha_1,\ldots,\alpha_{N_G}]$.
\end{proposition}
\begin{proof}
The proof uses the Jacobi identity (\ref{c27}) and is analogously to that of Lemma 31 of \cite{Br}. 
\end{proof}

\begin{corollary}
Assume that the edges $e_1,\ldots,e_n$ in $E(G)$ form a cycle. Then 
\begin{equation}\label{c42}
\ph^1=\sum_{j\neq 1} \lambda_j\ph^{1,j}\quad\text{with}\;\lambda_j=\pm 1. 
\end{equation}
\end{corollary}
\begin{proof}
The contraction of all of the edges of a cycle gives the vanishing of the dual graph polynomial:  $\ph^I_G=0$ for $I=\{1,\ldots,n\}$. Then (\ref{c40}) with $A=B=0$ implies the statement. 
\end{proof}

\begin{remark}
The corollary above implies the dual statement for the graph polynomial itself:  If the edges $e_1,\ldots,e_n$ form a cycle, then 
\begin{equation}
\Psi_1=\sum_{j\neq 1} \lambda_j\alpha_j\Psi^{1,j}\quad\text{with}\;\lambda_j=\pm 1. 
\end{equation}
Indeed, one only needs to apply the Cremona transformation to (\ref{c42}) and use Lemma \ref{lem10}. This statement was proved in \cite{BSY}, Propostion 24, using spanning forests polynomials. Our proof here is much more elementary.
\end{remark}
\begin{proposition}
Let $e_1,\ldots,e_n\in E(G)$ be the set of edges that form a corolla (have the same endpoint). Then 
\begin{equation}\label{c46}
\ph_1=\sum_{j\neq 1} \lambda_j\alpha_j\ph^{1,j}\quad\text{with}\;\lambda_j=\pm 1. 
\end{equation}
\end{proposition}
\begin{proof}
One obtains the result by dualizing the corresponding statement to (\ref{c42}) for graph polynomial (see \cite{BSY}, Remark 25) and by use of Lemma \ref{lem10}.
\end{proof}

\begin{example}\label{Extriang}
Consider a graph $G$ and assume that the edges $e_1$, $e_2$ and $e_3$ form a triangle. Choose the orientation of the triangle and orient the 3 edges in the corresponding way. Orienting the other edges arbitrarily, we fix a matrix $F_G$ and consider $\ph_G$ and the dual Dodgson polynomials. 

Since contraction of a loop leads to the vanishing of $\ph$, we get $\ph^{123}=0$. In addition to this, we also have $\ph^{12}_3=\ph^{23}_1=\ph^{13}_2$ since the deletion of one of the edges and contraction of the other two gives the same sub-quotient graph. The Jacobi identity (\ref{c27}) implies
\begin{equation}
\det \left( \begin{array}{ccc}
\ph^1 & \ph^{1,2} & \ph^{1,3} \\
\ph^{2,1} & \ph^2 & \ph^{2,3} \\
\ph^{3,1} & \ph^{3,2} & \ph^3 \end{array} \right) = 0.
\end{equation} 
By Proposition \ref{prop11}, we get $\ph^1=\ph^{1,2}-\ph^{1,3}$. These signs are fixed by the natural numeration of the edges in $F_G$. 
  
  Lets define $g_0:=\ph^{ij}_k$, $g_k:=(-1)^{j-i+1}\ph^{i,j}_k$, $g_{123}:=\ph_{123}$, where $\{i,j,k\}=\{1,2,3\}$.  

The identity above implies 
\begin{equation}
\ph^{12}_3\alpha_2+\ph^{13}_2\alpha_3+\ph^{1}_{23}= \ph^{13,23}\alpha_3+\ph^{1,2}_3+\ph^{12,23}\alpha_2-\ph^{1,3}_2.
\end{equation}
Working similarly with other rows of the matrix, we derive
\begin{equation}
g_0=\ph^{ij,jk}\quad \text{and}\quad \ph^{i}_{jk}=g_j+g_k.
\end{equation}
Now the dual graph polynomial $\ph_G$ takes the form
\begin{equation}\label{c51}
\ph_G=g_0(\alpha_1\alpha_2+\alpha_2\alpha_3+\alpha_1\alpha_3) + (g_2+g_3)\alpha_1+(g_1+g_3)\alpha_2+(g_1+g_2)\alpha_3  + g_{123}.
\end{equation}
By (\ref{c33}), we have the Dodgson identity $\ph^2_{13}\ph^3_{12}-\ph^{23}_1\ph_{123}=\ph^{2,3}\ph^{3,2}$. In our new notation, this reads $(g_1+g_3)(g_1+g_2)-g_0g_{123}=(g_1)^2$. Thus, 
\begin{equation}\label{c53}
g_0g_{123}=g_1g_2+g_2g_3+g_1g_3.
\end{equation}
The formulas in this example are identical to the case of a $3-$valent case for $\Psi_G$ in \cite{Br}, Example 32, and the situation is dual to the case of a triangle for $\Psi_G$ in Example 33, (loc.cit.).
\end{example}
Now we introduce our main geometrical object of interest. 
\begin{definition} 
For a graph $G$ with $N_G$ edges, define the dual graph hypersurface
\begin{equation} 
Z_G:=\cV(\ph_G)\subset\AAA^{N_G}_{\ZZ}.
\end{equation}
\end{definition}
Here we use the notation $\cV(f_1,\ldots,f_m)$ (resp. $\cV(\cI)$) to denote the variety defined by the vanishing of a set of polynomials $f_1,\ldots,f_m\in k[x_1,\ldots,x_n]$ (resp. of all the elements of an ideal $\cI$) in $\AAA^n$. The dimension of the ambient space is usually clear from the context. 

In the next sections we will try to understand the dual graph hypersurfaces by means of point-counting functions or the classes in the Grothendieck ring using the identities proved above.
\section{$K_0(Var_k)$ and $\FF_q$-rational points}
The essential results of the paper are formulated as some equalities and congruences between the numbers of $\FF_q$-rational points on the strata of the dual graph hypersurface. The stratification goes by successful elimination of the first few variables step by step in various orders. 

For a prime power $q$ and for an affine variety $Y$ defined over $\ZZ$, we define by $[Y]_q:=\# \bar{Y}(\FF_q)$ the number of $\FF_q$-rational points of $Y$ after extension of scalars to $\FF_q$. More o less, this means that $[f]_q$ is a number of solutions of $f=0$ in $\FF_q^n$ after taking the coefficients$\mod q$ for an affine hypersurface given by $f\in \ZZ[x_1,\ldots,x_n]$. Here and later, we use the shortcut $[f,\ldots,f_n]_q$ for $[\cV(f_1,\ldots,f_n)]_q$. We think of $[\cdot]_q$ as a function of $q$. Sometimes (but not in general), this function is a polynomial of $q$, when the arguments are some strata of the (dual) graph hypersurfaces discussed later for small graphs. Our main interest is the coefficient of $q^2$ of this function, the $c_2$ invariant (see (\ref{b2})). There are many graphs for which we know that $c_2(G)$ is constant, either 0 or 1, see \cite{BSY}. 

If one considers a closed subvariety $Y$ of a variety $X\subset\AAA^n_\ZZ$, then $[X]_q=[Y]_q+[X\backslash Y]_q$. This is the main relation we use during our computation. There is a space much more natural for this scissors relation, namely the Grothendieck ring of varieties over a field, $K_0(Var_k)$. It is much closer to the geometry of our varieties than the number of rational points. As a consequence, the computation on the level of $K_0(Var_k)$ can give us more information about poles of the Feynman differential form and about the period. The point-counting function factors through the Grothendieck ring. On the other hand, $K_0(Var_k)$ is very big, less tractable and, against our intuition, several reasonable statements analogues to the statements for point-counting functions fail in this ring, i.e. Chevalley-Warning vanishing. We prove a big part of our results on the level of the Grotendieck ring and shift to the computation of rational points only when we cannot avoid this. 

Define the Grothendieck ring of varieties over $k$, $K_0(Var_k)$ as the free group $\ZZ(Var_k)$ generated by all the varieties over $k$ after localization by the relation :  $X=Y+X\backslash Y$ for any variety $X$ and any closed subscheme $Y$. We denote by $[Y]$ the class of $Y$ in $K_0(Var_k)$. The ring structure is given by the product:  $[X]\cdot[Y]=[X\times_k Y]$. Define $\LL:=[\AAA^1]$ and $\1:=[Pt]$. 

By the factorisation of the point-counting function through the Grothendieck ring, $\LL$ is mapped to $q$ and $\1$ is mapped to 1. One can more or less ignore the use of $K_0(Var_k)$ just by thinking of our formulas as stated for the number of $\FF_q$-rational points by the substitution above.
  
In $K_0(Var_k)$, for the computation of the class of a variety given by the polynomials linear in one of the variables, one can try to eliminate that variable. The dual graph polynomial $\psi_G$ is linear in all of the variables, so the optimist may hope to get rid of the variables step by step.

\begin{lemma}\label{lemA17}
Let $f^1,f_1,g^1,g_1,h\in\ZZ[\alpha_2,\ldots,\alpha_n]$ be polynomials. Then, considering the varieties on the right hand side of the coming formulas to be in $\AAA^{n-1}$ and the varieties on the left to be in $\AAA^n$, 
\begin{enumerate}
\item for $f=f^1\alpha_1+f_1$, one has 
\begin{equation}\label{lin11}
[f,h]=[h] - [f^1,h] + [f^1,f_1,h]\LL,
\end{equation}
and, in particular, 
\begin{equation}\label{lin1}
[f]=\LL^{n-1} - [f^1] + [f^1,f_1]\LL.
\end{equation}
\item for $f=f^1\alpha_1+f_1$ and $g=g^1\alpha_1+g_1$, one has
\begin{equation}\label{lin21}
[f,g,h]=[f^1,f_1,g^1,g_1,h]\LL + [f^1g_1-g^1f_1,h] - [f^1,g^1,h].
\end{equation}
and
\begin{equation}\label{lin2}
[f,g]=[f^1,f_1,g^1,g_1] \LL+ [f^1g_1-g^1f_1] - [f^1,g^1] .
\end{equation}
\end{enumerate}
\end{lemma}
\begin{proof}
Equality (\ref{lin1}) follows from (\ref{lin11}) by putting $h=0$. For proving the equality (\ref{lin11}), consider the two cases $f^1=0$ and $f^1\neq 0$ separately. If $f^1=0$, then $f^1\alpha_1+f_1=0$ implies $f_1=0$ and $\alpha_1$ does not appear in the defining equations. So, $\cV(f,h)\cap\cV(f^1)$ is a trivial $\AAA^1$-fibration over $\cV(f^1,f_1,h)\subset\AAA^{n-1}$. If $f^1\neq 0$, then we evaluate $\alpha_1$ from $f^1\alpha_1+f_1=0$ and get an isomorphism between $\cV(h,f)\backslash \cV(h,f,f^1)\subset \AAA^n$ and $\cV(h)\backslash \cV(h,f^1)\subset \AAA^{n-1}$. One computes in $K_0(Var_k)$
\begin{equation}
[f,h]=[\AAA^1\times_k \cV(f^1,f_1,h)]+[\cV(h)\backslash \cV(h,f^1)].
\end{equation} 
Now the statement follows from the very definition of the classes in the Grothendieck ring. 

The equalities (\ref{lin21}) and (\ref{lin2}) are proved similarly by eliminating of the variable from the system and stratifying by the vanishing or non-vanishing of the coefficients. See, for example, \cite{Sch2}.
\end{proof}
\begin{proposition}\label{lemA18}
Let $G$ be a graph with $h_G\geq 2$. Then, the following holds:
\begin{description}
\item{\textbf{1)}} For some $c(G)\in K_0(Var_k)$
\begin{equation}
[Z_G] = c(G)\LL^2. 
\end{equation}
\item{\textbf{2)}} For some $b(G)\in K_0(Var_k)$ and for any edge $e_1$
\begin{equation}
[\ph^1_G,\ph_{G,1}] = b(G)\LL.
\end{equation}
\item{\textbf{3)}} For some $d(G)\in K_0(Var_k)$ and for any edges $e_1$, $e_2$
\begin{equation}
[\ph^{1,2}_G] = d(G)\LL.
\end{equation}
\end{description}
\end{proposition}
\begin{proof}
The proof goes by induction on $N_G$. The statements can be easily verified for graphs with $N_G\leq 3$. Assume that for all graphs with $N_G<M$ both parts (1) -- (3) were proved. Consider the case $N_G=M\geq 4$. We start with the class of $Z_G$ and eliminate the first variable by use of Lemma \ref{lemA17}, part 1) :
\begin{equation}\label{d20} 
[Z_G]=[\ph^1\alpha_1+\ph_1]=[\ph^1,\ph_1]\LL+\LL^{N-1}-[\ph^1].
\end{equation}
If $e_1$ is a self-loop, then $\ph^1=0$. Otherwise, $\ph^1$ itself a dual graph polynomial: $\ph^1=\ph_{G'}$ for $G'=G\q 1$ with $N_{G'}=N_G-1<M$, hence $[\ph^1]=c(G')\LL^2$ by the induction hypothesis. If part (2) holds for $N_G=M$, then part (1) also holds. Indeed, 
\begin{equation}
[Z_G]=b(G)\LL\cdot\LL+ \LL^{N_G-1} + c(G')\LL^2=(b(G)+c(G')+\LL^{N_G-3})\LL^2.
\end{equation} 
Now we prove part (2).

Both $\ph^1$ and $\ph_1$ are linear in the variable $\alpha_2$. Lemma \ref{lemA17}, part 2) 
allows us to get rid of $\alpha_2$ on $\cV(\ph^1,\ph_1)$:
\begin{multline}\label{d22}
[\ph^1,\ph_1]=[\ph^{12}\alpha_2+\ph^1_2,\ph^2_1\alpha_2+\ph_{12}]=\\ [\ph^{12},\ph^1_2,\ph^2_1,\ph_{12}]\LL+[\ph^1_2\ph^2_1-\ph^{12}\ph_{12}]-[\ph^{12},\ph^2_1].
\end{multline}
If $e_2$ is a self-loop, then both $\ph^{12}$ and $\ph^2_1$ are zero polynomials and the divisibility holds trivial. If $e_1$ and $e_2$ form a 2-cycle, then $\ph^{12}=0$ and $[\ph^2_{G,1}]=[\ph_{G'}]$ for $G'=G\backslash 1\q 2$ the last graph has again $h_G\geq 2$ and the divisibility follows from part (1), or has $h_G=1$ and the situation is easy to verify manually. Otherwise, for $e_1$ and $e_2$ in more general position,  $[\ph^{12}_G,\ph^2_{G,1}]=[\ph^{1}_{G'},\ph_{G',1}]$ with $G'=G\backslash 2$, it is divisible by $\LL$ by the induction hypothesis. By the first Dodgson identity, $[\ph^1_2\ph^2_1-\ph^{12}\ph_{12}]=[\ph^{1,2}]$.  If part (3) is proved, then 
\begin{equation}
b(G):=[\ph^{12},\ph^1_2,\ph^2_1,\ph_{12}]+d(G')-b(G').
\end{equation}
It remains to prove part (3).

Consider an edge $e_3$. If $\ph^{1,2}$ is independent of $\alpha_3$, then $\cV(\ph^{1,2})$ and divisibility is clear. Otherwise, we use Lemma \ref{lemA17} for $\alpha_3$: 
\begin{equation}
[\ph^{1,2}]= [\ph^{13,23}\alpha_3\pm\ph^{1,2}_3]\\ =\LL^{N-3}-[\ph^{13,23}]+[\ph^{13,23},\ph^{1,2}_3]\LL.
\end{equation}
Since $\ph^{13,23}_G=\ph^{1,2}_{G'}$ for $G'=G\backslash 3$, using the induction hypothesis, one computes 
\begin{equation}
d(G)=\LL^{N-4}-d(G')+[\ph^{13,23},\ph^{1,2}_3].
\end{equation}
This concludes the proof.
\end{proof}

\begin{remark}  As mentioned above, the equations for $[\ph_G]$ similar to that one in the proposition above give us equations for $[\ph_G]_q$ since point-counting functor factors through $K_0(Var_k)$ (or just by repeating all the steps). For example, under same conditions as in the proposition, $q^2|[\ph_G]_q$ and $q|[\ph^1_G,\ph_{G,1}]_q$.
\end{remark}
After the remark above, we are allowed to make the following definition. 
\begin{definition}
Let $G$ be a graph with $h_G\geq 2$. Define the $c_2$ invariant in dual parametric space: 
\begin{equation}
c_2(G)^{dual}:=[Z_G]_q/q^2 \mod q.
\end{equation}
\end{definition}
This $c_2$ invariant is the essential part the point-counting function, and, similar to $c_2(G)$ in (\ref{b2}), it satisfies many good properties. The most interesting of them is the coincidence of $c_2^{dual}(G)$ and $c_2(G)$ on the log-divergent graphs, see Theorem \ref{thm11} and Theorem \ref{thm53}. 
 
There is a more concrete description of the element $c(G)$ from Proposition \ref{lemA18}, if one has a cycle of length $\leq 3$. The most interesting case is when it is a cycle of length 3, a triangle.

Let $G$ be a graph with a triangle formed by the edges $e_1$, $e_2$, $e_3$. By Example \ref{Extriang}, the dual graph polynomial $\ph_G$ takes a form
\begin{equation}\label{d32}
\ph_G=g_0(\alpha_1\alpha_2+\alpha_2\alpha_3+\alpha_1\alpha_3) + (g_2+g_3)\alpha_1+(g_1+g_3)\alpha_2+(g_1+g_2)\alpha_3  + g_{123}.
\end{equation}
together with the connecting identity 
\begin{equation}\label{d33}
g_0g_{123}=g_1g_2+g_2g_3+g_1g_3.
\end{equation}  
\begin{proposition}
In the notation above, one has 
\begin{equation}
[Z_G]=\LL^{N_G-1}-\LL^2[g_0,g_1,g_2,g_3]+\LL^3[g_0,g_1,g_2,g_3,g_{123}].
\end{equation}
\end{proposition}
\begin{proof}
Since the formulas are identical to the case of the graph hypersurface $\Psi_G$ (but for 3-valent vertex), one can just repeat the proof of Proposition 23 in \cite{BrSch}. The proof is based of a geometrical argument on a related particular $\AAA^2$-fibration.
\end{proof}
\begin{proposition}\label{prop35}
Let $G$ be a graph with a triangle formed by $e_1$,$e_2$,$e_3$ with $h_G\geq 3$, $N_G\geq 4$. Then
\begin{equation}
[Z_G]\equiv [\ph^{13,23},\ph^{1,2}_3]\LL^2 \mod \LL^3.
\end{equation}
As a consequence,
\begin{equation}
[Z_G]_q\equiv q^2[\ph^{1,2}_3,\ph^{13,23}]_q \mod q^3,
\end{equation}
\end{proposition}
\begin{proof}
The proof is analogues to Lemma 24 in \cite{BrSch}, and identical to the part of the proof of Proposition 19 in \cite{D4}.
\end{proof}
\begin{remark}
After $c_2$ invariant $c_2^{dual}(G)$ is defined, Proposition \ref{prop35} gives a starting point for the \emph{denominator reduction} game similar to that one for computation of $c_2(G)$ in $\phi^4$ theory, see, for example, \cite{BrSch}. The set of graphs, for which this process will be applicable and give a concrete answer $c_2(G)=\pm 1$ or $0$, \emph{dually denominator reducible graphs}, need not to coincide with the set of denominator reducible graphs.
\end{remark}

Computing the number of rational points, we are also going to use the following vanishing statement called the \emph{Chevalley-Warning theorem}. This vanishing helps to get rid of many summands in the formulas coming later. 
\begin{theorem}\label{CW_thm}
Let $f_1,\ldots, f_k\subset \ZZ[x_1,\ldots,x_n]$ be polynomials and assume that the degrees $d_i:=\deg f_i$ satisfy $\sum_1^k d_i < n$. Then, for the number of $\FF_q$-rational points of the variety given by the intersection of the hyperplanes $\cV(f_i)$ in $\AAA^n$, the following congruence holds
\begin{equation}
  [f_1,\ldots,f_k]_q\equiv 0 \mod q.
\end{equation}  
\end{theorem}
\begin{proof}
The classical Chevalley-Warning statement was for $k=1$ and $q=p$. This was generalized to arbitrary prime power $q=p^m$ by Katz in \cite{Ka}. The general case easily follows by induction on $k$. \label{add_ref_c2invarinvar}
\end{proof}

The relevant to Feynman graphs case is the case of a log-divergent graph:  $N_G=2n_G$. For the further statements on log-divergent graphs, we need to understand that the situation $N_G>2n_G$ is "degenerate" for the point-counting function for $\ph_G$, more precisely, for the $c_2$ invariant. 

\begin{proposition}\label{lem19}
Let $G$ be a graph with $N_G>2n_G$. Assume $G$ has a triangle (resp. $G$ has a double edge or a self-loop and $n_G\geq 1$). Then the following congruences hold
\begin{align}
[Z_G]_g\equiv 0 \mod q^3,\label{al1}\\
[\ph^1,\ph_1]_q\equiv 0 \mod q^2\label{al2},
\end{align}
where $e_1$ is in the triangle (resp. double edge or self-loop).
\end{proposition}
\begin{proof}
The cases of a double edge and a self-edge are trivial.  Now, let $e_1$, $e_2$ and $e_3$ be the edges forming a triangle in $G$. By Proposition \ref{prop35}, $[Z_G]_q\equiv q^2[\ph^{1,2}_3,\ph^{13,23}]_q \mod q^3$. Now we are going to use Chevalley-Warning theorem. For this, we have to understand the degrees of the appearing polynomials. The degree of the dual graph polynomial is equal to number of vertices minus 1,  $\deg\ph_G=n_G$. By the first Dodgson identity, $\deg\ph^{i,j}=\deg\ph^i$. One computes $\deg \ph^{1,2}_3 = n_G-1$ and $\deg \ph^{13,23}= n_G-2$, both polynomials depend on $N_G-3$ variables. Since $N_G> 2n_G$, we may apply Chevalley-Warning theorem to $\cV(\ph^{1,2}_3,\ph^{13,23})$ and get 
\begin{equation}
[\ph^{1,2}_3,\ph^{13,23}]_q\equiv 0 \mod q.
\end{equation} 
The first statement follows.

For the second congruence, consider again the elimination of $\alpha_1$ by Lemma \ref{lemA17}:
\begin{equation}\label{d39} 
[Z_G]_q=[\ph^1\alpha_1+\ph_1]_q=q[\ph^1,\ph_1]_q+q^{N_G-1}-[\ph^1]_q.
\end{equation} 
Since $\ph^1_G=\ph_{G'}$ for $G'=G\q 1$ and $N_{G'}>2n_{G'}>1$ (or $n_{G'}=1$ and the situation is trivial). By the first statement, $[Z_G]_q\equiv [\ph^1]_q\equiv 0 \mod q^3$. Now $(\ref{d39})$ implies $q^2| [\ph^1,\ph_1]_q$.
\end{proof}

In \cite{D4}, it was proved that the $c_2$ invariant respects dualization (the coefficients of $q^2$ for $[Z_G]$ and for $\Psi_G$ coincide) for any log-divergent graph $G$ with $h_G\geq 3$ under the assumption that $G$ is duality admissible. 
\begin{definition}\label{def36}
A log-divergent graph $G$ with $n_G\geq 3$ (and $N=N_G=2n_G$ edges) is called $\underline{duality\; admissible}$ if 
\begin{equation}\label{d41}
[\ph^J_{I}]_q\equiv 0 \mod q^3
\end{equation}
for any $I,J\subset E(G)$ with $|J|>|I|\geq 0$, $|I|\leq n_G-3$.
\end{definition}

In the proof of the main result in \cite{D4}, the situation is symmetric under the interchange $\Psi_G\leftrightarrow \ph_G$. The vanishing corresponding to (\ref{d41}) for $\Psi_G$ is served by the statement similar to Proposition \ref{lem19} for a graph $G$ with $N_G>2h_G$ since such a graph always has vertex of valency at most $3$. The situation for $\ph_G$ is surprisingly more complicated since the girth of a (even log-divergent) graph is unbounded.

By the Proposition \ref{lem19} above , we know the divisibility of the point-counting functions for the sub-quotient graphs $[\ph^J_{I}]|q^3$ in the definition above as long as we have a cycle of length at most 3. In the next section we prove that the congruence (\ref{d41}) also holds in the case when we do not have a triangle, but have a 4-face.

\section{A 4-face formula}
In the previous section we have discussed several computational facts about the graphs with a cycles of length $\leq 3$. There are also graphs with $girth 4$, that is, all their cycles are of length $\geq 4$. Some of these graphs are relevant to this Feynman integrals subject, for example, known to give a counter-examples to the Kontsevich conjecture on the polynomiality of the point-counting function $[\Psi_G]_q$, see \cite{D2} or \cite{Sch}. On the other hand, we may meet such graphs when we are going to check the vanishing conditions (\ref{d41}) for all subgraphs while proving the duality admissibility for certain $G$.

In this section we try to study a graph $G$ with a 4-face in the similar way and with similar techniques as for the triangle case before.
\medskip

Consider a graph $G$ with a 4-face formed by the edges $e_1,\ldots,e_4$, with $e_1$ and $e_3$ opposite. What one can try do immediately is to start to reduce the first 2 variables by Lemma \ref{lemA17} and get 
\begin{equation}\label{d23}
[Z_G]=\LL^{N-2}-[\ph^1]+ [\ph^{12},\ph^1_2,\ph^2_1,\ph_{12}]\LL^2+ [\ph^{1,2}]\LL-[\ph^{12},\ph^2_1]\LL.
\end{equation}
The formula works for all graphs and the most complicated pies in the sum on the right is $\cV(\ph^{12},\ph^1_2,\ph^2_1,\ph_{12})$, the intersection of 4 hypersurfaces. This is also the obstruction for reducing the third variable in general. In the case $G$ having a triangle formed by the edges $e_1$,$e_2$ and $e_3$, one has a precise formula for $\ph_G$, see (\ref{c51}) in Example \ref{Extriang}. Using this, one derives $[Z_G]\equiv [\ph^{1,2},\ph^{13,23}]\mod \LL^3$, see Proposition \ref{prop35}. For the 4-face situation, we do not have the concrete formula for $\ph_G$ and, a priory, no such congruence. Nevertheless, we try to do our best, to understand the structure of $\ph_G$ and to prove some vanishing results similar to Proposition \ref{lem19}.
\medskip

We chose the orientation of the 4-face of $G$ and orient the edges $e_1,\ldots,e_4$ in the corresponding direction. Now we orient the other edges of $G$ and build the matrix $L_G$ to fix the signs of the Dodgson polynomials. The contraction of the edges $e_1,\ldots,e_4$ leads to the contraction of a self-loop, hence $\ph^{1234}=0$. We also know that $\ph^{123}_4=\ph^{ijk}_t$ for $\{i,j,k,t\}=\{1,2,3,4\}$.
Similarly to Example \ref{Extriang}, the Jacobi identity (\ref{c27}) implies the vanishing of the corresponding $4\times 4$ matrix. The first row implies
\begin{equation}
\ph^{1,1}=\ph^{1,2}-\ph^{1,3}+\ph^{1,4}.
\end{equation}
Expanding these polynomials in $\alpha_2$, $\alpha_3$ and $\alpha_4$, one gets
\begin{multline}
\Psi^{123}_4\alpha_2\alpha_3+ \ph^{124}_3\alpha_2\alpha_4+ \ph^{134}_2\alpha_3\alpha_4+ \ph^{12}_{34}\alpha_2+ \ph^{13}_{24}\alpha_3+ \ph^{14}_{23}\alpha_4+
\ph^{1}_{234}=\\ (\ph^{134,234}\alpha_3\alpha_4+ \ph^{13,23}_{4}\alpha_3+ \ph^{14,24}_{3}\alpha_4+\ph^{1,2}_{34}) - (-\ph^{124,234}\alpha_2\alpha_4-\ph^{12,23}_{4}\alpha_2+\\ \ph^{14,34}_{2}\alpha_4+\ph^{1,3}_{24})+ (\ph^{123,234}\alpha_2\alpha_3- \ph^{12,24}_{3}\alpha_2- \ph^{13,34}_{2}\alpha_3+\ph^{1,4}_{23})
\end{multline}
We derive $\ph^{123}_4=\ph^{123,234}$, $\ph^{12}_{34}=\ph^{12,23}_4-\ph^{12,24}_3$, $\ph^{1}_{234}=\ph^{1,2}_{34}-\ph^{1,3}_{24}+\ph^{1,4}_{23}$.
For unifying the notation, define 
\begin{equation}
a:=\ph^{ijk}_t ,\quad c^{i,j}:=(-1)^{i-j-1}\ph^{i,j},\quad b^i_j:=(-1)^{r_b}\ph^{ki,it}_j,
\end{equation}
where $r_b=(k-t)$ if $(k-i)(t-i)>0$, and $r_b=(k-t-1)$ otherwise.
Analysing similarly the other rows of the matrix, we finally obtain
\begin{equation}\label{b14}
\begin{aligned}
\ph^{ijk,ijt}&=a=\ph^{ijk}_t,\\ 
\ph^{ij}_{kt}&=b^i_k+b^i_t,\\
\ph^i_{jkt}&=c^{i,j}+c^{i,k}+c^{i,t},
\end{aligned}
\end{equation}
for all $\{i,j,k,t\}=\{1,2,3,4\}$. The polynomials are also related by Dodgson identities.
Applying the formula before (\ref{c53}) to the case $G'=G\backslash t$, we get in $\ZZ[\alpha]$
\begin{equation}\label{b15}
(b^i_t)^2\equiv \ph^{ij}_{kt}\ph^{ik}_{jt} \mod a.
\end{equation}
Now we return to formula (\ref{d23}). 

\bigskip 
To get (partial) control on the class of $\cV(\ph^{12},\ph^1_2,\ph^2_1,\ph_{12})$, we are going to stratify this intersection further by reducing with respect to the next 2 variables using Dodgson identities and the identities from (\ref{b14}).
\begin{theorem}\label{prop51}
Let $G$ be a graph with a 4-face bounded by the edges $e_1,\ldots,e_4$, where $e_1$ and $e_3$ be opposite edges. Then 
\begin{multline}\label{b28}
[\ph^{12},\ph^1_2,\ph^2_1,\ph_{12}]\equiv [\ph^{12,34}]-[a,\ph^{12,34}]+\\  [a,b^1_3]-[a,b^1_4] + [a,\ph^{12}_{34}\ph^{34}_{12}] \mod \LL.
\end{multline}
\end{theorem}
\begin{proof}
Recall the formula for eliminating of one variable $\alpha=\alpha_1$ from the set of polynomials $f_1,\ldots,f_k\in \ZZ[\alpha_1,\ldots,\alpha_n]$ linear in this variable, $f_i=f_i^1\alpha+f_{i,1}$ :
\begin{multline}\label{b29}
[f_1,\ldots,f_n]=[f^\alpha_1,f_{1,\alpha},\ldots,f^\alpha_n,f_{n,\alpha}]\LL +\\ [[f_1,f_2]_\alpha,\ldots,[f_1,f_n]_\alpha] - [f^\alpha_1,\ldots,f^\alpha_n]\\
\sum^{n-2}_{k=1}([f^\alpha_1,f_{1,\alpha}\ldots,f^\alpha_k,f_{k,\alpha},[f_{k+1},f_{k+2}]_{\alpha},\ldots,[f_{k+1},f_n]_{\alpha}]\\
-[f^\alpha_1,f_{1,\alpha}\ldots,f^\alpha_k,f_{k,\alpha}]).
\end{multline}
see \cite{BSY}, Proposition 29.\\ 
Here and later, for two polynomials $f$ and $g$ linear of  $\alpha_i$, we denote by $[f,g]_{\alpha_i}=[f,g]_i$ the resultant with respect to $\alpha_i$:
\begin{equation}
[f,g]_i:=\pm (f^ig_i-f_ig^i).
\end{equation}
We apply formula (\ref{b29}) to the polynomials 
\begin{equation}
f_a=\ph^{12},\; f_b=\ph^1_2,\; f_c=\ph^2_1\; f_d=\ph_{12}
\end{equation}
for the variable $\alpha=\alpha_3$. Then we get 
\begin{multline}\label{b30}
[\ph^{12},\ph^1_2,\ph^2_1,\ph_{12}]=
[f_a,f_b,f_c,f_d]=[f^3_a,f_{a3},f^3_b,f_{b3},f^3_c,f_{c3},f^3_d,f_{d3}]\LL\\ + \big(S_1+S_2+S_3\big) - \big([f^3_a,f^3_b,f^3_c,f^3_d]+[f^3_a,f_{a3}]+[f^3_a,f_{a3},f^3_b,f_{b3}]\big),
\end{multline}
where
\begin{equation}\label{b31}
\begin{aligned}
S_1=\big[[f_a,f_b]_3,[f_a,f_c]_3,[f_a,f_d]_3\big],\\
S_2=\big[f^3_a,f_{a3},[f_b,f_c]_3,[f_b,f_d]_3\big],\\
S_3=\big[f^3_a,f_{a3},f^3_b,f_{b3},[f_c,f_d]_3\big].
\end{aligned}
\end{equation}
Each of the three summands in the last brackets of (\ref{b30}) is divisible by $\LL$. Indeed, the variety $\cV(f^3_a,f^3_b,f^3_c,f^3_d)\subset\AAA^{N-2}$ is the cone over the variety defined by the same equations but in $\AAA^{N-3} (\text{no}\; \alpha_3)$, thus $\LL|[f^3_a,f^3_b,f^3_c,f^3_d]$. Now $[f_a^3,f_{a3}]=[\ph^{123},\ph^{12}_3]=[\ph^3_{G'},\ph_{G',3}]$ for $G'=G\q 12$, so $\LL |[f_a^3,f_{a3}]$ by Proposition \ref{lemA18}. For the last summand $[f^3_a,f_{a3},f^3_b,f_{b3}]=[\ph^{123},\ph^{12}_3,\ph^{13}_2,\ph^1_{23}]$ we are going to use the triangle formula from Example \ref{Extriang} for the graph $G':=G\q 1$ with edges $e_2,e_3,e_4$ forming a triangle. In the notation with $g_i$ but with indices $i=2,3,4$, we have 
\begin{multline}\label{d35}
[\ph^{23}_{G'},\ph^{2}_{G',3},\ph^{3}_{G',2},\ph_{G',23}]=[g_0,g_0\alpha_3+(g_3+g_4),g_0\alpha_3+(g_2+g_3),\\ (g_2+g_4)\alpha_3+g_{234}]=[g_0,g_3+g_4,g_2+g_3,(g_2+g_4)\alpha_3+g_{234}].
\end{multline}
The connecting identity (\ref{c53}) takes the form $g_0g_{234}=g_2(g_3+g_4)+g_3g_4$, thus the vanishing of $g_3+g_4$ on $\cV(g_0)$ implies the vanishing of both summands $g_3$ and $g_4$. Analogously, 
\begin{equation}\label{b32}
[g_0,g_2+g_3]=[g_0,g_2,g_3].
\end{equation}
It follows now that all the terms in the brackets 
(\ref{d35}) become independent of $\alpha_3$. As a consequence, it gives us a cone over a variety in $\AAA^{N-3}$, thus the class is divisible by $\LL$.

Finally, we derive the following congruence from (\ref{b30}):
\begin{equation}\label{d36}
[\ph^{12},\ph^1_2,\ph^2_1,\ph_{12}]\equiv\big(S_1+S_2+S_3\big) \mod \LL
\end{equation} 
with $S_i$ given by (\ref{b31}). Now we will work with these 3 summands separately and then will show that they sum up to $0 \mod \LL$. For simplicity, we list here the involved polynomials:
\begin{equation}\label{d37}
\begin{aligned}
\mathstrut[f_a,f_b]_3&=\ph^{123}\ph^1_{23}-\ph^{12}_3\ph^{13}_2=(\ph^{12,13})^2=(a\alpha+b^1_4)^2,\\
[f_a,f_c]_3&=\ph^{123}\ph^2_{13}-\ph^{12}_3\ph^{23}_1=(\ph^{12,23})^2=(a\alpha+b^2_4)^2,\\
[f_c,f_d]_3&=\ph^{23}_1\ph_{123}-\ph^2_{13}\ph^3_{12}=(\ph^{2,3}_1)^2,\\
[f_b,f_d]_3&=\ph^{13}_2\ph_{123}-\ph^3_{12}\ph^1_{23}=(\ph^{1,3}_2)^2,\\
[f_b,f_c]_3&=\ph^{13}_2\ph^2_{13}-\ph^{23}_1\ph^1_{23},\\
[f_a,f_d]_3&=\ph^{123}\ph_{123}-\ph^{12}_3\ph^3_{12}.
\end{aligned}
\end{equation}

The coefficient of $\alpha_2$ in the expansion of the first Dodgson identity $\ph^1_3\ph_1^3-\ph^{13}\ph_{13}=(\ph^{1,3})^2$ in $\alpha_2$ gives 
\begin{equation}\label{b34}
\ph^{12}_3\ph^3_{12}+\ph^1_{23}\ph^{23}_1- \ph^{123}\ph_{123}-\ph^{13}_2\ph^2_{13} = -2\ph^{12,23}\ph^{1,3}_2.
\end{equation} 
Similarly, for the expansion in $\alpha_1$ of the Dodgson identity for the pair of edges $e_2$ and $e_3$ implies
\begin{equation}\label{b35}
\ph^{12}_3\ph^3_{12}+\ph^2_{13}\ph^{13}_2- \ph^{123}\ph_{123}-\ph^{23}_1\ph^1_{23} = 2\ph^{12,13}\ph^{2,3}_1.
\end{equation} 
The sum of the two equalities above reads 
\begin{equation}\label{b36}
\ph^{12}_3\ph^3_{12}-\ph^{123}\ph_{123}= \ph^{12,13}\ph^{2,3}_1-\ph^{12,23}\ph^{1,3}_2. 
\end{equation} 
It follows that $[f_a,f_d]_3\in \ZZ[\alpha]$ is in the ideal generated by $\ph^{12,13}$ and $\ph^{12,23}$. Thus, using (\ref{d37}), one computes
\begin{multline}\label{b37}
S_1=[[f_a,f_b]_3,[f_a,f_c]_3,[f_a,f_d]_3]=[\ph^{12,13},\ph^{12,23}]=\\ [a\alpha_4+b^1_4,a\alpha_4+b^2_4]=[a\alpha_4+b^1_4,b^2_4-b^1_4].
\end{multline}
Similar to Lemma \ref{lem7}, by use of the classical Pl\"uker identity, we can derive the following identity on the minors of $L_G$ in (\ref{d2}):
\begin{equation}
\det L_G(1,2,3,4)-\det L_G(1,2,1,3)+\det L_G(1,2,2,3)=0.
\end{equation}
The expansion in $\alpha_4$ gives
\begin{equation}\label{b39}
\ph^{12,34}=b^2_4-b^1_4.
\end{equation}
After the elimination of $\alpha_4$ by (\ref{lin11}),  the equalities (\ref{b37}) and (\ref{b39}) imply
\begin{equation}\label{b40}
S_1\equiv [\ph^{12,34}]-[a,\ph^{12,34}] \mod \LL.
\end{equation}
Now we are going to compute $S_2$:
\begin{equation}
S_2=[f^3_a,f_{a3},[f_b,f_c]_3,[f_b,f_d]_3]=[a,\ph^{12}_{34},[f_b,f_c]_3,\ph^{1,3}_2]. 
\end{equation}
We use again the equalities (\ref{b34}) and (\ref{b35}) and now subtract instead of adding. We immediately get
\begin{equation}
[f_b,f_c]_3=[\ph^{13}_2\ph^2_{13}-\ph^{23}_1\ph^1_{23}]= \ph^{12,13}\ph^{2,3}_1+\ph^{12,23}\ph^{1,3}_2.
\end{equation}
It follows that 
\begin{multline}
S_2=[a,\ph^{12}_{34},\ph^{1,3}_2,\ph^{12,13}\ph^{2,3}_1]=[a,\ph^{12}_{34},\ph^{1,3}_2,(a\alpha_4+b^1_4)\ph^{2,3}_1]\\
=[a,\ph^{12}_{34},b^4_2\alpha_4+\ph^{1,3}_{24},b^1_4\ph^{2,3}_1].
\end{multline} 
The last term of the last brackets disappears, this follows from (\ref{b15}): $b^1_4$ vanishes on $\cV(a,\ph^{12}_{34})$.
By (\ref{lin11}), eliminating $\alpha_4$, one now computes 
\begin{equation}
S_2\equiv [a,\ph^{12}_{34}] - [a,\ph^{12}_{34},b^4_2]\mod \LL.
\end{equation}
The third summand of (\ref{d36}), $S_3$, takes the form  
\begin{multline}\label{b42}
S_3=[f^3_a,f_{a3},f^3_b,f_{b3},[f_c,f_d]_3]=[a,\ph^{12}_{34},\ph^{13}_{24},\ph^1_{23},\ph^{2,3}_1]=\\ [a,\ph^{12}_{34},\ph^{13}_{24},\ph^{14}_{23}\alpha_4+\ph^1_{234},b^4_1\alpha_4+\ph^{2,3}_{14}].
\end{multline}
We claim that $\ph^{14}_{23}$ lies in the ideal generated by $a,\ph^{12}_{34},\ph^{13}_{24}$. Indeed, $\ph^{14}_{23}=b^1_2+b^1_3$ and, by (\ref{b32}), $b^1_2$ vanishes on $\cV(a,\ph^{13}_{24})$ while $b^1_3$ vanishes on $\cV(a,\ph^{12}_{34})$. Thus only the last polynomial in (\ref{b42}) depends on $\alpha_4$. One computes 
\begin{equation}
S_3\equiv [a,\ph^{12}_{34},\ph^{13}_{24},\ph^1_{234}]-[a,\ph^{12}_{34},\ph^{13}_{24},\ph^1_{234},b^4_1]\mod \LL.
\end{equation}  
Consider the equation similar to (\ref{b36}) but for the collection of edges $(e_1,e_2,e_4)$ instead of $(e_3,e_1,e_2)$: 
\begin{equation}
\ph^{24}_1\ph^1_{24}-\ph^{124}\ph_{124}= \ph^{14,24}\ph^{1,2}_4-\ph^{12,24}\ph^{1,4}_2.
\end{equation}
Each of the appearing polynomials depends on $\alpha_3$. A consideration of the constant coefficient gives
\begin{equation}\label{b45}
\ph^{24}_{13}\ph^1_{234}-a\ph_{1234}=b^4_3\ph^{1,2}_{34}-b^2_3\ph^{1,4}_{23}.
\end{equation}
Consider the variety $Z=\cV(a,\ph^{12}_{34},\ph^{13}_{24})\subset \AAA^{N_G-4}$ and let $Y=Z\backslash Z\cap \cV(b^4_1)$. Since the vanishing of $\ph^{12}_{34}$ implies $b^1_3=0$ and the vanishing of $\ph^{13}_{24}$ implies $b^1_2=0$ on $\cV(a)$ by (\ref{b15}), one gets also $\ph^{14}_{23}=b^1_2+b^1_3=0$ on $\cV(a)$. Hence, again by (\ref{b15}), $b^4_3$ vanishes on $Z$. The equation (\ref{b45}) now implies $\ph^{24}_{13}\ph^1_{234}=0$ on $Z$. Since $\ph^{24}_{13}=b^4_1+b^4_3$, and $b^4_3=0$ while $b^4_1\neq 0$ on $Y$, one derives $Y\cap \cV(\ph^1_{234})\cong Y$. Thus $S_3=[\cV(a,\ph^{12}_{34},\ph^{13}_{24},\ph^1_{234})\backslash \cV(a,\ph^{12}_{34},\ph^{13}_{24},\ph^1_{234},b^4_1)]=[Y]$. One computes 
\begin{multline}\label{b47}
S_2+S_3\equiv ([a,\ph^{12}_{34}] + [a,\ph^{12}_{34},\ph^{13}_{24}])\\ -( [a,\ph^{12}_{34},b^4_2]+ [a,\ph^{12}_{34},\ph^{13}_{24},b^4_1]) \mod \LL.
\end{multline}
For the third summand, one uses the equality  $(b^4_2)^2\equiv \ph^{14}_{23}\ph^{34}_{12}\mod a$ in (\ref{b15}) and gets
\begin{multline}\label{b48}
[a,\ph^{12}_{34},b^4_2]=[a,\ph^{12}_{34},\ph^{14}_{23}\ph^{34}_{12}]=[a,\ph^{12}_{34},\ph^{14}_{23}]+[a,\ph^{12}_{34},\ph^{34}_{12}]\\
-[a,\ph^{12}_{34},\ph^{14}_{23},\ph^{34}_{12}].
\end{multline}
Similarly, 
\begin{multline}\label{b49}
[a,\ph^{12}_{34},\ph^{13}_{24},b^4_1]=[a,\ph^{12}_{34},\ph^{13}_{24},\ph^{24}_{13}\ph^{34}_{12}]=[a,\ph^{12}_{34},\ph^{13}_{24},\ph^{24}_{13}]+\\
[a,\ph^{12}_{34},\ph^{13}_{24},\ph^{34}_{12}]-[a,\ph^{12}_{34},\ph^{13}_{24},\ph^{24}_{13},\ph^{34}_{12}].
\end{multline}
The last summands of (\ref{b48}) and (\ref{b49}) coincide. Indeed, $\ph^{12}_{34}=0=\ph^{13}_{24}$ on $\cV(a)$ imply $\ph^{23}_{14}=0$ since $e_1,e_2,e_3$ form a triangle in $G\q 4$, and also $\ph^{23}_{14}=0=\ph^{34}_{12}$ imply $\ph^{24}_{13}=0$ in the triangle $e_2,e_3,e_4$ in $G\q 1$. One derives
\begin{multline}
[a,\ph^{12}_{34},\ph^{13}_{24},\ph^{24}_{13},\ph^{34}_{12}]=[a,\ph^{12}_{34},\ph^{13}_{24},\ph^{34}_{12}]=[a,\ph^{12}_{34},\ph^{13}_{24},\ph^{24}_{13}]\\
=[a,\ph^{12}_{34},\ph^{14}_{23},\ph^{34}_{12}].
\end{multline}
Hence, 
\begin{multline}\label{b51}
S_2+S_3\equiv [a,\ph^{12}_{34}] + [a,\ph^{12}_{34},\ph^{13}_{24}]-[a,\ph^{12}_{34},\ph^{14}_{23}]\\-[a,\ph^{12}_{34},\ph^{34}_{12}] \mod \LL.
\end{multline}
The first summand on the right hand side is divisible by $\LL$ by Proposition \ref{lemA18} applied to $[\ph^{3}_{G'},\ph_{G',3}]$ for $G'=G\backslash 4\q \{1,2\}$. Similarly, the second summand on the right hand side of the equality 
\begin{equation}\label{b52}
[a,\ph^{12}_{34},\ph^{13}_{24}]=[a,\ph^{12}_{34}]+[a,\ph^{13}_{24}]-[a,\ph^{12}_{34}\ph^{13}_{24}]
\end{equation}
is divisible by $\LL$. Using the equality (\ref{b15}), one gets 
\begin{equation}
[a,\ph^{12}_{34},\ph^{13}_{24}]\equiv -[a,\ph^{12}_{34}\ph^{13}_{24}]\equiv -[a,b^1_4] \mod \LL.
\end{equation}
The same thing can be done with $[a,\ph^{12}_{34},\ph^{14}_{23}]$ in (\ref{b51}). One can also do the step (\ref{b52}) for $[a,\ph^{12}_{34},\ph^{34}_{12}]$. The congruence (\ref{b51}) now implies
\begin{equation}
S_2+S_3\equiv [a,b^1_3]-[a,b^1_4] + [a,\ph^{12}_{34}\ph^{34}_{12}] \mod \LL. 
\end{equation}
By (\ref{d36}) and (\ref{b40}), we finally get the desired formula
\begin{multline}
[\ph^{12},\ph^1_2,\ph^2_1,\ph_{12}]\equiv S_1+S_2+S_3\equiv [\ph^{12,34}]-[a,\ph^{12,34}]\\ +[a,b^1_3]-[a,b^1_4] + [a,\ph^{12}_{34}\ph^{34}_{12}] \mod \LL.
\end{multline}
\end{proof}

What we mean a \emph{4-face formula} is just the ability to express the class $[Z_G]\!\mod \LL^3$ in the formula (\ref{d23}) by use of classes of the intersections of up to 3 hypersurafaces, after Theorem \ref{prop51}. It is possible to write down a more concrete formula on the level of point-counting function for, say, log-divergent graphs, but this does not lead to new results. Nevertheless, the very important application of the technique above is the following result:

\begin{proposition} \label{Thm42}
Let $G$ be a graph with $N_G\geq 2n_G$ and assume it has a 4-face. Let $e_1$ and $e_2$ be two adjacent edges of a 4-cycle bounding this face. Then
\begin{equation}\label{b58}
[\ph^{12},\ph^1_2,\ph^2_1,\ph_{12}]_q\equiv 0 \mod q.
\end{equation}
\end{proposition}
\begin{proof}
Denote by $e_3$ and $e_4$ the two other edges of the named 4-cycle going in the natural ordering. Consider a graph $G'$ to be the following modification of $G$ : we delete 4 first edges, introduce 2 new edges $e_s$ and $e_t$ instead, and identify 2 vertices as shown on Figure 1. 
\begin{figure}[h]
\centering
\begin{picture}(250,120)
\thicklines
\put(10,50){\line(2,-3){30}}
\put(10,50){\line(2,3){30}}
\put(10,50){\line(-2,-1){10}}
\put(10,50){\line(-1,0){10}}
\put(10,50){\line(-2,1){10}}
\put(70,50){\line(-2,-3){30}}
\put(70,50){\line(-2,3){30}}
\put(70,50){\line(2,-1){10}}
\put(70,50){\line(1,0){10}}
\put(70,50){\line(2,1){10}}
\put(40,95){\line(1,2){5}}
\put(40,95){\line(0,1){10}}
\put(40,95){\line(-1,2){5}}
\put(40,5){\line(-1,-2){5}}
\put(40,5){\line(0,-1){10}}
\put(40,5){\line(1,-2){5}}
\put(10,50){\circle*{3}}
\put(70,50){\circle*{3}}
\put(40,95){\circle*{3}}
\put(40,5){\circle*{3}}
\put(12,77){\text{$e_1$}}\put(56,77){\text{$e_2$}} \put(45,33){\text{$e_3$}}\put(12,20){\text{$e_4$}}
\put(68,10){\text{$G$}}
\put(110,35){\vector(1,0){40}}
\put(190,50){\line(0,1){45}}
\put(190,50){\line(0,-1){45}}
\put(190,50){\line(-2,-1){10}}
\put(190,50){\line(-1,0){10}}
\put(190,50){\line(-2,1){10}}
\put(190,50){\line(2,-1){10}}
\put(190,50){\line(1,0){10}}
\put(190,50){\line(2,1){10}}
\put(190,95){\line(0,1){10}}
\put(190,95){\line(1,2){5}}
\put(190,95){\line(-1,2){5}}
\put(190,5){\line(0,-1){10}}
\put(190,5){\line(-1,-2){5}}
\put(190,5){\line(1,-2){5}}
\put(190,50){\circle*{3}}
\put(190,5){\circle*{3}}
\put(190,95){\circle*{3}}\put(210,10){\text{$G'$}}
\put(175,20){\text{$e_t$}}\put(175,70){\text{$e_s$}}
\thinlines 
\end{picture}
\caption{From $G$ to $G'$.}
\end{figure}
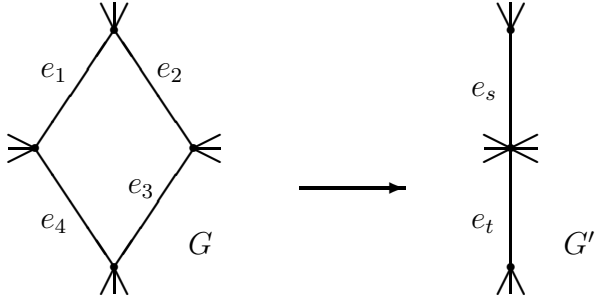\\
It has $N_G-2$ edges $e_s,e_t,e_5,e_6, \ldots,e_{N_G}$ and $n_{G'}=n_G-1$. One immediately sees that 
\begin{equation}
\ph^{12}_{G,34}=\ph^s_{G',t}\quad \text{and}\quad \ph^{34}_{G,12}=\ph^t_{G',s}.
\end{equation}
Using the first Dodgson identity for $I=\{a\}$, $J=\{b\}$, one gets
\begin{equation}\label{b60}
\cV(a,\ph^{12}_{G,34}\ph^{34}_{G,12})\cong\cV(\ph^{st}_{G'},\ph^s_{G',t}\ph^t_{G',s})\cong\cV(\ph^{st}_{G'},\ph^{s,t}_{G'}).
\end{equation}
Since the point-counting functor factors through the Grothendieck ring, (\ref{b28}) implies the following congruence:
\begin{multline}\label{b61}
[\ph^{12},\ph^1_2,\ph^2_1,\ph_{12}]_q\equiv [\ph^{12,34}]_q-[a,\ph^{12,34}]_q+\\ [a,b^1_3]_q-[a,b^1_4]_q + [\ph^{st}_{G'},\ph^{s,t}_{G'}]_q \mod q.
\end{multline}
One computes the degrees:
\begin{equation}
\deg b^i_j=\deg \ph^{12,34}_{G'}=\deg \ph^{s,t}_{G'}=n_G-2,\quad \deg a=\deg \ph^{st}_{G'}=n_G-3. 
\end{equation}
Since all of the varieties in (\ref{b61}) are considered to be in $\AAA^{N_G-4}$, and $N_G\geq 2n_G$, Chevalley-Warning theorem  implies the vanishing of all the summands on the right hand side. Hence
\begin{equation}
[\ph^{12},\ph^1_2,\ph^2_1,\ph_{12}]_q \equiv 0 \mod q.
\end{equation}
\end{proof}
Proposition \ref{Thm42} gives us some control on $[\ph^{12},\ph^1_2,\ph^2_1,\ph_{12}]_q$, this can help to compute $c_2^{dual}(G)$. Returning to Formula (\ref{d23}), we also want to "understand" the summand $[\ph^{1,2}]_q$ in this sense. 
\begin{lemma}\label{lemm43}
Let $G$ be a graph with $N_G> 2n_G$ having a 4-face. Let $e_1$ and $e_2$ be two adjacent edges bounding this 4-face. Then
\begin{equation}
[\ph^{1,2}]_q \equiv 0 \mod q^2.
\end{equation}
\end{lemma}
\begin{proof} By Lemma \ref{lemA17}, we can get rid of the variables $\alpha_3$ and $\alpha_4$ :
\begin{multline}\label{b70}
[\ph^{1,2}]=[\ph^{13,23}\alpha_3+\ph^{1,2}_3]= \LL^{N-3}-[\ph^{13,23}]+\LL[\ph^{13,23},\ph^{1,2}_3]\\ =\LL^{N-3}-[\ph^{13,23}]+\LL^2[\ph^{134,234}, \ph^{13,23}_{4}, \ph^{14,24}_{3},\ph^{1,2}_{34}]\\ +\LL[\ph^{134,234}\ph^{1,2}_{34}-\ph^{13,23}_{4}\ph^{14,24}_{3}] -\LL[\ph^{134,234},\ph^{14,24}_{3}].
\end{multline}
Applying the first Dodgson identity again (just to get a nicer form) and then appling the Chevalley-Warning theorem, we obtain
\begin{equation}\label{b71}
[\ph^{134,234}\ph^{1,2}_{34}-\ph^{13,23}_{4}\ph^{14,24}_{3}]_q\equiv [\ph^{13,24}\ph^{14,23}]_q\equiv 0 \mod q
\end{equation}
since we are dealing with a product of total degree $2(n_G-2)=2n_G-4$ of $N_G-4$ variables and $N_G>2n_G$ by the assumption. Next, the application of the Chevalley-Warning theorem also implies
\begin{equation}\label{b72}
[\ph^{134,234},\ph^{14,24}_{3}]_q \equiv 0 \mod q.
\end{equation}
By Lemma \ref{lemA17}, we compute
\begin{multline}\label{b73}
[\ph^{13,23}]_q=[\ph^{134,234}\alpha_4+\ph^{13,23}_4]_q=q^{N_G-4}-[\ph^{134,234}]_q\\
+q[\ph^{134,234},\ph^{13,23}_4]_q \equiv 0 \mod q^2,
\end{multline}
here we have again used the Chevalley-Warning vanishing for the last summand and also Proposition \ref{lemA18}, part (1) for $\cV(\ph^{134,234})$. 
Now (\ref{b70}) together with (\ref{b71})-(\ref{b73}) imply
the desired congruence.
\end{proof}

Now we are ready to prove the main theorem about the structure of $[Z_G]_q$ in the 4-face case.

\begin{theorem}\label{thm54}
Let $G$ be a graph with $N_G>2n_G$. Assume $G$ has a 4-face. Then 
\begin{equation}
[Z_G]_q\equiv 0 \mod q^3.
\end{equation}
\end{theorem}
\begin{proof}
The equality (\ref{d23}) in the Grothendieck ring implies the corresponding equality for the point-counting functions:
\begin{equation}\label{b75}
[Z_G]_q=q^{N_G-2}-[\ph^1]_q+ q^2[\ph^{12},\ph^1_2,\ph^2_1,\ph_{12}]_q+ q[\ph^{1,2}]_q-q[\ph^{12},\ph^2_1]_q.
\end{equation}
The graph $G'=G\q 1$ has a triangle formed by the edges $e_2$, $e_3$, $e_4$, and one has $N_{G'}>2n_{G'}$. By Proposition \ref{lem19}, 
\begin{equation}\label{b76}
[\ph^1_G]_q =[\ph_{G'}]\equiv 0 \mod q^3.
\end{equation}
The variety $\cV(\ph^{12},\ph^2_1)$ is isomorphic to $\cV(\ph^{1}_{G'},\ph_{G',1})$ for $G''=G\q 2$. The graph $G''$ has a triangle formed by the edges $e_1$, $e_3$, $e_4$, it satisfies $N_{G''}>2n_{G''}$. Proposition \ref{lem19} is again applicable:
\begin{equation}\label{b77}
[\ph^{12}_G,\ph^2_{G,1}]_q \equiv [\ph^{1}_{G''},\ph_{G'',1}]_q \equiv 0 \mod q.
\end{equation}
By Proposition \ref{Thm42} and Lemma \ref{lemm43}, one also has
\begin{equation}\label{b78}
[\ph^{1,2}]_q\equiv q[\ph^{12},\ph^1_2,\ph^2_1,\ph_{12}]_q\equiv 0 \mod q^2.
\end{equation} 
The substitution of (\ref{b76}) -- (\ref{b78}) into (\ref{b75}) implies the statement.
\end{proof}

We can also derive a short formula for the $c_2$ invariant in the case $G$ being log-divergent. This is what we call a \emph{4-face formula} for $c_2^{dual}(G)$.
\begin{theorem}
Let $G$ be a log-divergent graph ($N_G=2n_G$) with a 4-face bounded by the edges $e_1,\ldots,e_4$. Then 
\begin{equation}\label{b81}
c_2^{dual}(G)\equiv -[\ph^{13,24},\ph^{14,23}]_q \mod q.
\end{equation}
\end{theorem}
\begin{proof}
By (\ref{b58}),   we know the congruence $[\ph^{12},\ph^1_2,\ph^2_1,\ph_{12}]_q\equiv 0\mod q$ for a log-divergent graph $G$. Thus, (\ref{b75}) yields 
\begin{equation}\label{b82}
[Z_G]_q\equiv q[\ph^{1,2}]_q-[\ph^1]_q-q[\ph^{12},\ph^2_1]_q\mod q.
\end{equation} 
Since $[\ph^{12},\ph^2_1]_q=[\ph^1_{G'},\ph_{G',1}]_q$ for $G'=G\q 2$, a graph with a triangle and $N_{G'}>2n_{G'}$, Proposition \ref{lem19} implies $[\ph^{12},\ph^2_1]_q\equiv 0\mod q$. Similarly, $[\ph^1]_q\equiv 0 \mod q^2$.
  
In the proof of Lemma \ref{lemm43}, the only term in the right hand side of (\ref{b70}) that survives mod $q$ for a log-divergent $G$ is the term from (\ref{b71}). Thus, 
\begin{equation}
[Z_G]_q\equiv q^2[\ph^{13,24}\ph^{14,23}]_q \equiv - q^2[\ph^{13,24},\ph^{14,23}]_q \mod q^3.
\end{equation}
\end{proof}
The formula above is, up to a sign, independent of the numeration of the 4 edges. This can be easily shown using (\ref{c25}).

\section{Girth 5 and Conclusion}

Recall that $girth(G)$ is the minimal $n$ such that each cycle of $G$ is of length $\geq n$. In general, $\girth(G)$ is unbounded. Even if we restrict to $\phi^4$ or to log-divergent graphs, it is not very difficult to construct examples of graphs of any given girth.

To establish that a graph is duality admissible (see Definition \ref{def36}), one needs to check the vanishing condition:
\begin{equation}\label{b85}
[\ph_{G'}]_q\equiv 0 \mod q^3
\end{equation} 
for all sub-quotient graphs $G'=G\backslash I\q J$ for any $I,J\subset E(G)$ with $|J|>|I|\geq 0$, $|I|\leq n_G-3$. If $G'$ has a cycle of length at most 3, then the vanishing follows from Proposition \ref{lem19}   since $N_{G'}=N_G-|I|-|J|>2(n_G-|J|)=n_{G'}$. If $G'$ does not have a triangle, but does have a cycle of length 4, then we again obtain the congruence (\ref{b85}) by Theorem \ref{thm54}. On the other side, if the minimal cycle in the graph $G'$ is on length $\geq 4$, we cannot prove the congruence. The absence of a 3-face and 4-face is an obstruction to our methods. We need to estimate the minimal $N_G$ for which this situation can occur.

A nice (and most physically interesting) situation is the case when a graph $G$ is log-divergent in $\phi^4$ theory. That is, it is obtained from the 4-regular graph $\hat{G}$ (all the vertices are 4-valent) after deletion of one of the vertices. The graph $\hat{G}$ is called the completion of $G$. There is an interesting arithmetic conjecture about the graphs with the same completion, see Conjecture 4 in \cite{BrSch}. 

We recall a well-known result of Robertson, \cite{R}:

\begin{theorem}[Robertson]\label{thm_Rob}
There is a 4-regular graph with $girth=5$ and 19 vertices. It is the unique (up to isomorphism) graph with these properties among all graphs with less than 20 vertices. 
\end{theorem}

Let $\hat{R}$ be the Robertson's graph above. Then the corresponding $R$ is a log-divergent graph of $girth(G)=5$. It is a graph with minimal $N_R$  with these conditions. It has  $h_G=n_G=17$, $N_G=34$. What we need is a slightly different thing.
\begin{lemma}\label{lem62}
Let $G$ be a graph with $3\leq n_G\leq 17$ and with $N_G>2n_G$. Then $\girth(G)<5$.
\end{lemma}
\begin{proof}
The proof is done with the help of a computer. To optimize the brute force, one can start similarly to the proof from \cite{R}. Assume that there exists such a graph with $\girth 5$. If $G$ has a 5-valent vertex $v$, one can consider the arcs (paths) of length 2 from $v$. The endpoints ($k$ up to $n_G-5$) should be mutually different and they are connected by $N_G-k-5$ edges. One has  several possibilities and can find a contradiction in a few steps. Now when all the vertices are up to 4-valent, we proceed with a small exhaustive search on a PC.  
\end{proof}

We are ready to state our main theorem.
\begin{theorem}\label{thm53}
Let $G$ be a log-divergent graph with $3\leq h_G\leq 18$ loops. Then $G$ is duality admissible. 
\end{theorem}
\begin{proof}
Consider any relevant sub-quotient graph $G':=G\backslash I \q J$, see the Definition \ref{def36}. Then $G'$ has $n_{G'}\leq n_G-1=h_G-1\leq 17$ and $N_{G'}> 2n_{G'}$ edges. Now Lemma \ref{lem62} implies that $G'$ has a cycle of length at most $4$. As was already explained above, under this assumption Proposition \ref{lem19} or Theorem \ref{thm54} provide the needed congruence
\begin{equation}
[Z_{G'}]_q\equiv 0 \mod q^3.
\end{equation}
This concludes the proof.
\end{proof}
As a consequence, we finally get 
\begin{theorem}\label{thm5n4}
Let $G$ be a log-divergent graph with $3\leq h_G\leq 18$. Then all the $c_2$ invariants in all four different representations of the Feynman period coincide: 
\begin{equation}\label{b110}
c_2(G)^{mom}_q=c_2(G)_q=c_2(G)^{dual}_q=c_2(G)^{pos}_q.
\end{equation} 
\end{theorem}

This follows from the results of \cite{D4}. This is again an indication that $c_2$ invariant is a good discrete analogue to the Feynman period. The range of $h_G$ is more than enough and covers all physically relevant graphs.

Nevertheless, Theorem \ref{thm54} also proves the equality (\ref{b110}) for a larger set of graphs, since the graphs of girth 5 occur rather rare. We formulate the result as a combinatorial sufficient condition.

\begin{theorem}\label{thm5n5}
Let $G$ be a graphs with $h_G\geq 3$. If each sub-quotient graph $\gamma=G\backslash I\q J$, where $I,J\subset E(G)$, $|J|>|I|\geq 0$, $|I|\leq n_G-3$, has a loop of length at most 4, then all 4 $c_2$ invariants coincide.
\end{theorem}

I believe that there exists a 5-face formula or even n-face formula with the similar meaning: even for a log-divergent graph with a big girth, the most complicated summand $[\ph^{12},\ph^1_2,\ph^2_1,\ph_{12}]_q$ can be killed mod $q$, and, after (\ref{b82}),
the $c_2$ invariant $c^{dual}_2$ can be computed naturally by the corresponding step of the denominator reduction while the total contribution of the other summands is zero.

\end{document}